\documentclass[12pt]{article}
\usepackage{latexsym, amsfonts, amssymb, amscd, amsthm}
\usepackage{amsmath, mathrsfs, yhmath}
\usepackage{epsfig, graphicx, subfig}
\usepackage{hyperref}
\newtheorem{theorem}{Theorem}[section]

\newtheorem{lemma}[theorem]{Lemma}

\newtheorem{corollary}[theorem]{Corollary}

\newtheorem{remark}[theorem]{Remark}
\numberwithin{equation}{section}
\allowdisplaybreaks[4]

\title{\bf %The behaviour of
Star-shaped Centrosymmetric Curves under Gage's Area-preserving Flow}

\author{ \bf  Laiyuan Gao~~~~Shengliang Pan}
\date{}

\begin{document}
\maketitle
\noindent{\bf Abstract}
It is proved that Gage's area-preserving flow can evolve a centrosymmetric %and
star-shaped initial curve smoothly, make it convex in a finite time and deform it into
a circle as time tends to infinity.
%with an initial centrosymmetrically embedded and star-shaped curve exists globally,  drives the evolving curve being


\section{Introduction}

A regular closed embedded plane curve %in the Euclidean plane
can be parameterized as $X: S^1\rightarrow
\mathbb{R}^2, \varphi\mapsto (x(\varphi), y(\varphi))$. If the total length of $X$ is $L$, $X$ is usually
parameterized as $X: [0, L] \rightarrow\mathbb{R}^2, s\mapsto (x(s), y(s))$, where $s$ is the arc length parameter.
%Let $X_0: S^1\rightarrow \mathbb{R}^2 (\varphi\mapsto (x, y))$ be a smooth, closed and embedded curve in the
%Euclidean plane. Let $s\in [0, L]$ be the arc length parameter, where $L$ is the length of the curve.
Denote by $T(s)$ the unit tangential vector at a point $X(s)$ and by $N(s)=N_{in}(s)$ the inward unit
normal vector such that every ordered pair $(T(s), N(s))$ determines a positive orientation of the plane.
The curve $X$ is called \textbf{star-shaped} if there exists a point $O$ inside the region bounded by $X$
such that
$$\det (X(s), T(s)) = \begin{vmatrix}
x(s) & y(s) \\
x^{\prime}(s) & y^{\prime}(s)
\end{vmatrix} >0$$
for
every $s$, %in the Cartesian coordinate system $(x, y; O)$.
and the point $O$ is called a star center of the curve. If $X$ is a $C^2$ curve then its (relative) curvature is defined as
$$\kappa(s) := \left\langle \frac{\partial T}{\partial s}(s), ~N(s)\right\rangle.$$

%Let $X_0$ be a star-shaped smooth curve in the plane.
In this paper, we will investigate the evolution behaviour of a star-shaped
centrosymmetric plane curve under Gage's {\bf area-preserving flow (GAPF)} \cite{Gage-1986},
that is to say, we will consider the following Cauchy problem
\begin{equation}\label{eq:1.1.201909}
\left\{\begin{array}{l}
\frac{\partial X}{\partial t}(\varphi, t) =
\left(\kappa-\frac{2\pi}{L}\right)N \ \ \ \text{in} \ \ S^1\times (0, \omega),\\
X(\varphi, 0)= X_0(\varphi) \ \  \ \ \ \  \text{on} \ \ S^1,
\end{array} \right.
\end{equation}
where $X: S^1\times [0, \omega)\rightarrow \mathbb{R}^2 ((\varphi, t)\mapsto (x, y))$ is a family of
closed curves with $X_0$ being star-shaped and centrosymmetric about the origin $O$ of the plane,
$\kappa=\kappa(\varphi, t)$ the curvature and
$L=L(t)$ the length of $X(\cdot, t)$. In 1984, Gage (\cite{Gage-1986}) proved that the evolving
curve $X(\cdot, t)$ under this flow %(\ref{eq:1.1.201909})
can be deformed into a circle if the initial
curve $X_0$ is convex. As Gage pointed out that there is a kind of simple closed curves which may
develop singularities in a finite time under the flow (\ref{eq:1.1.201909}), which is verified by
Mayer's numerical experiment (see \cite{Mayer-2001}). % did a  to support Gage's claim.
So, unlike Grayson's theorem \cite{Grayson-1987} for the {\bf curve shortening flow (CSF)},
there is no convergence result of GAPF for generic simple closed initial curves.
A natural question is whether there is a class of non-convex initial curves which may become convex and
converge to a circle under the flow (\ref{eq:1.1.201909}).
To guarantee such a convergence, referring to Gage's example (see \cite{Gage-1986, Mayer-2001}),
we should assume that the initial curve $X_0$ can not concave wildly. In this paper,
we consider $X_0$ to be star-shaped %condition of  is considered
and centrosymmetric and obtain the following main result:
\begin{theorem}\label{thm:1.1.201909}
Let $X_0$ be a smooth, embedded  and star-shaped curve in the plane.
If $X_0$ is centrosymmetric then Gage's area-preserving flow (\ref{eq:1.1.201909}) with this
initial curve exists globally and makes the evolving curve convex in a finite time
and deforms it into a circle as time tends to infinity.
\end{theorem}

A flow %(\ref{eq:1.1.201909})
is called global if the evolving curve is smooth for all $t\in [0, +\infty)$.
For an initial star-shaped curve in the plane, it is still unknown whether Gage's area-preserving flow
exists globally or not (see Lemma \ref{lem:2.5.201909} and Remark \ref{rem:3.4.201909}).
If $X_0$ is only star-shaped, %without other conditions,
it seems that the flow (\ref{eq:1.1.201909}) may not preserve the star-shapedness of the evolving
curve during the evolution process. It
causes essential difficulties to understand the asymptotic behavior of $X(\cdot, t)$.
The extra symmetric property of $X_0$ is inspired by Prof. Michael Gage \cite{PC-2018}
(also see the early work \cite{Gage-1993, Gage-Li-1994} by Gage and Li).

The proof of Theorem \ref{thm:1.1.201909} is divided into two parts. In %the proof of
the first part, i.e. the global existence of the flow,
the star-shapedness of the evolving curve plays an essential role,
see Lemma \ref{lem:3.7.201909} and Corollary \ref{cor:3.8.201909}. To show
that $X(\cdot, t)$ is star-shaped, $X_0$ needs to be of centrosymmetry. % of $X_0$ is needed. Suppose
If $X_0$ is centrosymmetric with respect to the point $O$, the
key idea in this part is to show that the evolving curve never touches the point $O$
via a comparison to the evolution behavior of the famous CSF,
see Lemmas \ref{lem:3.1.201909}-\ref{lem:3.6.201909} for the detaails. In %the proof of
the second part, i.e. the convergence of the evolving curve,
some ideas in Grayson's papers \cite{Grayson-1987, Grayson-1989} are adopted, see
Lemmas \ref{lem:4.2.201909}-\ref{lem:4.3.201909}.

Once $X(\cdot, t)$ is proved to be star-shaped, the polar angle $\theta$ can be used
as a parameter of the evolving curve. In order to make $\theta$ independent of time, one can add a tangent
component to the origin flow to get a new one:
\begin{equation}\label{eq:1.2.201909}
\left\{\begin{array}{l}
\frac{\partial X}{\partial t}(\varphi, t) = \alpha T + \left(\kappa - \frac{2\pi}{L}\right)N
\ \ \ \text{in} \ \ S^1\times (0, \omega),\\
X(\varphi, 0)= X_0(\varphi) \ \  \ \ \ \  \text{on} \ \ S^1.
\end{array} \right.
\end{equation}
The tangent component $\alpha(\varphi, t)T(\varphi, t)$
(not influence the shape of $X$, see Proposition 1.1 on page 6 of \cite{Chou-Zhu}) will be determined
in the next section.

GAPF \cite{Gage-1986} has also been considered by Wang, Wo and Yang \cite{Wang-Wo-Yang-2018} if the initial curve
is closed and locally convex. They have studied some asymptotic behaviours of the evolving curve,
including the convergence for global flows and some blow-up properties.
For higher dimensional cases, %the relative work of higher dimensions,
one can refer to Huisken's volume-preserving flow of convex hypersurfaces \cite{Huisken-1987}
and its generalization by Kim and Kwon \cite{Kim-Kwon-2018} to the case of star-shaped hypersurfaces with
a so called $\rho$-reflection property.

This paper is organized as follows. In Section 2, some basic properties of the flow (\ref{eq:1.1.201909})
are obtained, including the short time existence and a property of $X(\cdot, t)$ which implies its star-shapeness.
In Section 3, it is  proved that the flow (\ref{eq:1.1.201909}) exists on the time interval $[0, +\infty)$.
And in the final section, the proof of Theorem \ref{thm:1.1.201909} is completed. % in Section 4.

\section{Preparation}
\setcounter{equation}{0}

Given a curve $X$ in the plane, its ``support function" is defined by\footnote{$p$ is usually
called the support function of $X$ in convex geometry, see, for example, \cite{Gru-2007, Hs-1981, Sch-2014}.}
$$ p=-\langle X, N\rangle.$$
If the curve is expressed as  $X(s)=(x(s), y(s))$, $s\in[0,L]$, where $L$ is the length of $X$, then its
unit tangent and normal vector fields can be written as
$$T(s)=(\dot{x}(s), \dot{y}(s)), \ \ \ \ \ \ \  N(s)=(-\dot{y}(s), \dot{x}(s)), $$
where ``$\cdot$" stands for derivative with respect to the arc length parameter $s$. Since
$$p(s)=x(s)\dot{y}(s)-\dot{x}(s)y(s)=\det (X(s), T(s)),$$
$X$ is star-shaped with respect to the origin $O$
if and only if $p(s)>0$ for all $s$.

Using the polar coordinate system $(r, \theta)$ for the plane, a smooth and closed curve can be expressed as
$$X(s)=r(\theta)P(\theta),$$
where $\theta=\theta(s)$ and $P(\theta)=(\cos\theta, \sin\theta)$. Let $Q(\theta)=(-\sin\theta, \cos\theta)$,
then the unit tangential vector of $X$ can be given by $$T = \frac{dr}{ds}P +
r\frac{d\theta}{ds} Q,$$
and furthermore
\begin{equation}\label{eq:2.1.201909}
\det \langle X(s), T(s)\rangle=r^2(s) \dot{\theta}(s).%\frac{\partial \theta}{\partial s}.
\end{equation}
If $X$ is closed and star-shaped then one can choose the origin $O$ of our frame so that $r(s)>0$
and $\det \langle X, T\rangle > 0$. The equation (\ref{eq:2.1.201909}) implies that one can use the polar
angle $\theta$ to parameterize a star-shaped plane curve.

Now let us deal with the flow (\ref{eq:1.1.201909}) with a star-shaped initial value $X_0$.
We shall first derive some evolution equations and determine the tangential component $\alpha T$
to make the polar angle $\theta$ independent
of time $t$. Then (\ref{eq:1.1.201909}) can be reduced to a Cauchy problem of a single equation
for the radial function $r=r(\theta, t)$. After that, some basic properties of the flow
(\ref{eq:1.1.201909}) will be explored in this section.

Let $g:=\sqrt{\langle \frac{\partial X}{\partial \varphi}, \frac{\partial X}{\partial \varphi}\rangle}$
be the metric of the evolving curve. Set
$\beta=\kappa-\frac{2\pi}{L}$. Under the flow (\ref{eq:1.2.201909}), $g$ evolves according to
\begin{eqnarray*}
\frac{\partial g}{\partial t}
= \frac{1}{g}\left\langle\frac{\partial}{\partial t}\frac{\partial X}{\partial \varphi}, \frac{\partial X}{\partial \varphi}\right\rangle
=g\left\langle\frac{\partial}{\partial s}(\alpha T+\beta N), T\right\rangle
=\left(\frac{\partial\alpha}{\partial s}-\beta\kappa\right)g.
\end{eqnarray*}
The interchange of the operators $\partial/\partial s$ and $\partial/\partial t$ is given by
\begin{eqnarray*}
\frac{\partial}{\partial t}\frac{\partial}{\partial s}=\frac{\partial}{\partial t}\left(\frac{1}{g}\frac{\partial}{\partial \varphi}\right)
=\frac{\partial}{\partial s}\frac{\partial}{\partial t}-\left(\frac{\partial\alpha}{\partial s}-\beta\kappa\right)\frac{\partial}{\partial s}.
\end{eqnarray*}
$T$ and $N$ evolve according to
\begin{eqnarray*}
&&\frac{\partial T}{\partial t}=\frac{\partial }{\partial t}\frac{\partial X}{\partial s}
=\frac{\partial }{\partial s}\frac{\partial X}{\partial t}-\left(\frac{\partial\alpha}{\partial s}-\beta\kappa\right)T
=\left(\alpha\kappa+\frac{\partial\beta}{\partial s}\right)N,
\\
&&\frac{\partial N}{\partial t}=\left\langle\frac{\partial N}{\partial t}, T\right\rangle T
+\left\langle\frac{\partial N}{\partial t}, N\right\rangle N
=-\left(\alpha\kappa+\frac{\partial\beta}{\partial s}\right)T.
\end{eqnarray*}
If there is a family of star-shaped curves evolving under the flow (\ref{eq:1.2.201909}) then we can express the evolving curve as
\begin{eqnarray}\label{eq:2.2.201909}
X(\theta, t)=r(\theta, t)P(\theta).
\end{eqnarray}
Noticing that $\frac{\partial X}{\partial \theta}=\frac{\partial r}{\partial \theta}P+rQ$, we obtain
\begin{eqnarray}\label{eq:2.3.201909}
g=\left\|\frac{\partial X}{\partial \theta}\right\|=\left(r^2+\left(\frac{\partial r}{\partial \theta}\right)^2\right)^{1/2}, ~~
T=\frac{\partial r}{\partial s}P+\frac{r}{g}Q, ~~ N=-\frac{r}{g}P+\frac{\partial r}{\partial s}Q.
\end{eqnarray}
Differentiating the right hand side of (\ref{eq:2.2.201909}) and using (\ref{eq:1.2.201909}) and (\ref{eq:2.3.201909}), one gets
\begin{eqnarray*}
\frac{\partial r}{\partial t}P+r\frac{\partial \theta}{\partial t}Q=\alpha T+\beta N
=\left(\alpha\frac{\partial r}{\partial s}-\frac{r\beta}{g}\right)P
+\left(\frac{\alpha r}{g}+\beta\frac{\partial r}{\partial s}\right) Q.
\end{eqnarray*}
Comparing the coefficients of both sides can yield the following evolution equations:
\begin{eqnarray}\label{eq:2.4.201909}
\frac{\partial r}{\partial t}=\alpha\frac{\partial r}{\partial s}-\frac{r\beta}{g}, \ \ \
\frac{\partial \theta}{\partial t}=\frac{\alpha}{g}+\frac{\beta}{r}\frac{\partial r}{\partial s}.
\end{eqnarray}
From now on, we choose $$\alpha=-\frac{\beta}{r}\frac{\partial r}{\partial s}g=-\frac{\beta}{r}\frac{\partial r}{\partial \theta}$$
so that $\frac{\partial \theta}{\partial t}\equiv 0$, the polar angle $\theta$ is independent of the time $t$.
%\begin{eqnarray*} \frac{\partial \theta}{\partial t}\equiv 0. \end{eqnarray*}
Since the curvature of the evolving curve is
$$\kappa=\frac{1}{g^3}\left(-r\frac{\partial^2 r}{\partial \theta^2}+2\left(\frac{\partial r}{\partial \theta}\right)^2+r^2\right),$$
one can immediately obtain the evolution equation of $r$,
\begin{eqnarray}\label{eq:2.5.201909}
\frac{\partial r}{\partial t}=\frac{1}{g^2} \frac{\partial^2 r}{\partial \theta^2}-\frac{2}{rg^2}\left(\frac{\partial r}{\partial \theta}\right)^2
-\frac{r}{g^2}+\frac{2\pi g}{rL}.
\end{eqnarray}

Now, %On the other hand,
if $r=r(\theta, t)>0$ is defined on $[0, 2\pi]\times [0, \omega)$ and satisfies
the equation (\ref{eq:2.5.201909}), then a family of curves
$\{X=rP|t\in [0, \omega)\}$ satisfies the flow (\ref{eq:1.2.201909}). So we can reduce the flow (\ref{eq:1.2.201909}) to
the equation (\ref{eq:2.5.201909}) with initial value $r_0(\theta)>0$:
%Since the equation (\ref{eq:2.5.201909}) is strictly
%parabolic, it has a unique smooth solution on a short time interval.
\begin{lemma}\label{lem:2.1.201909}
Suppose $X_0$ is star-shaped. The flow (\ref{eq:1.2.201909}) is equivalent to the
quasi-linear parabolic equation (\ref{eq:2.5.201909})
with a positive initial value $r_0$ in some interval $[0, \omega)$.
\end{lemma}

The length of the curve can be calculated according to %It follows from the definition of the length
$$L(t)=\int_0^{2\pi}g(\theta, t) d\theta=\int_0^{2\pi}\sqrt{r^2+\left(\frac{\partial r}{\partial \theta}\right)^2} d\theta.$$
Let us define an operator $F$ from the space
$C^{2, \alpha}([0, 2\pi]\times [0, \omega))$ to $C^{\beta}([0, 2\pi]\times [0, \omega))$, for $0<\beta<\alpha\leq 1$ according to
$$F(r)=\frac{\partial r}{\partial t}-
\frac{1}{g^2} \frac{\partial^2 r}{\partial \theta^2}
+\frac{2}{rg^2}\left(\frac{\partial r}{\partial \theta}\right)^2
+\frac{r}{g^2}-\frac{2\pi g}{rL}.$$
Since the Frechet derivative of $F$ at some point $r_0>0$ is
$$DF(r_0)f=\frac{\partial f}{\partial t}-\frac{1}{r_0^2+(\frac{\partial r_0}{\partial\theta})^2}\frac{\partial^2 f}{\partial\theta^2}
+\text{lower linear terms of}\ f,$$
the equation (\ref{eq:2.5.201909}) is uniformly parabolic near its initial value $r_0$. It follows from the
implicit function theorem of Banach spaces that
the Cauchy problem (\ref{eq:2.5.201909}) has a unique solution in some small time interval (See Section 1.2 in \cite{Chou-Zhu}).
Using Lemma \ref{lem:2.1.201909} can give us the short time existence.
\begin{lemma}\label{lem:2.2.201909}
The flow (\ref{eq:1.1.201909})  has a unique smooth solution in some time interval $[0, \omega)$, where $\omega>0$.
\end{lemma}

Next, we shall derive some basic properties of the flow (\ref{eq:1.1.201909}).

\begin{lemma}\label{lem:2.3}
Under the flow (\ref{eq:1.1.201909}), the area $A$ of the evolving curve is constant, that is, $A(t)\equiv A_0$;
and the length $L$ satisfies $\sqrt{4\pi A_0}\leq L(t)\leq L_0$.
\end{lemma}
\begin{proof}
This result is a direct consequence of %Using
the equations (1.18)-(1.19) in \cite{Chou-Zhu} and %, one obtains %that
%\begin{eqnarray*} \frac{d A}{d t}=-\int_0^L (\kappa-\frac{2\pi}{L}) ds=0, \\
%\frac{d L}{d t}=-\int_0^L \kappa^2 ds+\frac{4\pi^2}{L}\leq 0. \end{eqnarray*}
%The last inequality follows from the Cauchy-Schwarz inequality. So the area enclosed by $X$ is invariant and the length $L$ is decreasing,
%which together with
the classical isoperimetric inequality.
\end{proof}

Under the flow (\ref{eq:1.2.201909}), the ``support function" is
$$p=-\langle X, N\rangle=-\left\langle rP, -\frac{r}{g}P+\frac{1}{g}\frac{\partial r}{\partial \theta} Q\right\rangle =\frac{r^2}{g},$$
and thus a closed curve is star-shaped if and only if $r>0$ and $\left|\frac{\partial r}{\partial \theta}\right|$ is bounded everywhere.
Since $r$ and $\left|\frac{\partial r}{\partial \theta}\right|$ satisfy parabolic equations, one can apply the comparison principle
to bound these two functions. For a continuous function $f=f(\theta, t)$, set
$$f_{\min}(t)=\min\{f(\theta, t)|\theta\in [0, 2\pi]\}, ~~f_{\max}(t)=\max\{f(\theta, t)|\theta\in [0, 2\pi]\}.$$

\begin{lemma}\label{lem:2.4.201909}
Given a star-shaped curve $X_0$, we choose a point $O$ in the plane such that $r_0(\theta)>0$. If $r(\theta, t)\geq c>0$ for
$(\theta, t)\in [0, 2\pi]\times [0, \omega)$ under the flow (\ref{eq:1.2.201909}) then
\begin{eqnarray}
r(\theta, t) \leq \frac{L_0}{2}, \label{eq:2.6.201909}\\
\left|\frac{\partial r}{\partial \theta}(\theta, t) \right|\leq C_1, \label{eq:2.7.201909}
\end{eqnarray}
where $C_1=\max\left\{\max\limits_{\theta}\left|\frac{\partial r}{\partial \theta}(\theta, 0)\right|, \frac{3L_0}{\pi}\right\}$
is a constant depending on the initial curve $X_0$.
\end{lemma}
\begin{proof}
Fix a $t_0\in [0, \omega)$. Let $t\in[0, t_0)$, if $r(\cdot, t)$ attains its maximum value
$r_{\max}(t)$ at $(\theta_*, t)$, then
$$\frac{\partial^2 r}{\partial \theta^2} (\theta_*, t)\leq 0, ~~~
\frac{\partial r}{\partial \theta}(\theta_*, t)=0,~~~
g(\theta_*, t)=r(\theta_*, t).$$
By (\ref{eq:2.5.201909}),
\begin{eqnarray*}
\frac{\partial r}{\partial t}(\theta_*, t) \leq \frac{2\pi}{L(t)}\leq \sqrt{\frac{\pi}{A_0}}.
\end{eqnarray*}
So the maximum principle implies that if $t\in [0, t_0)$ then
\begin{eqnarray}\label{eq:2.8.201909}
r(\theta, t) \leq r_{\max}(0)+ \sqrt{\frac{\pi}{A_0}} t_0.
\end{eqnarray}

Differentiating the equation (\ref{eq:2.5.201909}), one gets
\begin{eqnarray*}
\frac{\partial^2 r}{\partial t \partial \theta}&=&\frac{1}{g^2}\frac{\partial^3 r}{\partial \theta^3}
   -\frac{2}{g^4}\frac{\partial r}{\partial \theta} \left(\frac{\partial^2 r}{\partial \theta^2}\right)^2
   -\frac{4}{rg^2}\frac{\partial r}{\partial \theta}\frac{\partial^2 r}{\partial \theta^2}
   +\frac{2}{r^2g^2}\left(\frac{\partial r}{\partial \theta}\right)^3
   \\
&& +\frac{4}{g^4}\left(\frac{\partial r}{\partial \theta}\right)^3
   +\frac{4}{rg^4}\left(\frac{\partial r}{\partial \theta}\right)^3 \frac{\partial^2 r}{\partial \theta^2}
   -\frac{1}{g^2}\frac{\partial r}{\partial \theta}
   +\frac{2r^2}{g^4}\frac{\partial r}{\partial \theta}
\\
&& +\frac{2\pi}{Lg}\frac{\partial r}{\partial \theta}
   +\frac{2\pi}{rLg} \frac{\partial r}{\partial \theta} \frac{\partial^2 r}{\partial \theta^2}
   -\frac{2\pi g}{r^2L} \frac{\partial r}{\partial \theta}.
\end{eqnarray*}
Let $w=(\frac{\partial r}{\partial \theta})^2$, it evolves according to
\begin{eqnarray*}
\frac{1}{2}\frac{\partial w}{\partial t} &=& \frac{1}{2}\frac{1}{g^2}\frac{\partial^2 w}{\partial \theta^2}
   -\frac{1}{g^2}\left(\frac{\partial^2 r}{\partial \theta^2}\right)^2
   -\frac{2w}{g^4}\left(\frac{\partial^2 r}{\partial \theta^2}\right)^2
   -\frac{2}{rg^2}\frac{\partial r}{\partial \theta}\frac{\partial w}{\partial \theta}
\\
&& +\frac{2w^2}{r^2g^2}
   +\frac{4w^2}{g^4} +\frac{2w}{rg^4}\frac{\partial r}{\partial \theta}\frac{\partial w}{\partial \theta}
   -\frac{w}{g^2}+\frac{2r^2w}{g^4}
 \\
&& +\frac{2\pi w}{Lg}+\frac{\pi}{rLg}\frac{\partial r}{\partial \theta}\frac{\partial w}{\partial \theta}
   -\frac{2\pi gw}{r^2L}.
\end{eqnarray*}
If $w$ attains its maximum value $w_{\max}(t_0)$ at $(\theta_*, t_0)$, then
$\frac{\partial w}{\partial \theta}(\theta_*, t_0)=0, \frac{\partial^2 w}{\partial \theta^2}(\theta_*, t_0)
\leq 0,$ and thus at the point $(\theta_*, t_0)$, one obtains
\begin{eqnarray*}
\frac{1}{2}\frac{\partial w}{\partial t} &\leq& \left(\frac{2w^2}{r^2g^2}+\frac{4w^2}{g^4}\right)+
\left(-\frac{w}{g^2} +\frac{2r^2w}{g^4}\right)
+\left(\frac{2\pi w}{Lg}-\frac{2\pi gw}{r^2L}\right)
\\
&=&\frac{2w^2(3r^2+w)}{r^2g^4}+\frac{w}{g^4}(r^2-w)-\frac{2\pi w^2}{r^2Lg}
\\
&=& \frac{2w^2(3Lr^2+wL)-2\pi w^2(r^2+w)g}{r^2Lg^4}+\frac{w}{g^4}(r^2-w)
\\
&=& \frac{2 w^2r^2(3L-\pi g)+2w^3(L-\pi g)}{r^2Lg^4}+\frac{w}{g^4}(r^2-w).
\end{eqnarray*}
If $w\geq \max\{\left(r_{\max}(t)\right)^2, \left(\frac{3L_0}{\pi}\right)^2\}$
then $3L-\pi g \leq 0$ and $r^2-w\leq 0$. One gets
$\frac{\partial w}{\partial t}(\theta_*, t_0)\leq 0$.
The maximum principle tells us
\begin{eqnarray}\label{eq:2.9.201909}
\left|\frac{\partial r}{\partial \theta}\right|
\leq \max\left\{\max_{\theta}\left|\frac{\partial r}{\partial \theta}(\theta, 0)\right|,
~r_{\max}(t),
~\frac{3L_0}{\pi}\right\}.
\end{eqnarray}

Combining (\ref{eq:2.8.201909}) and (\ref{eq:2.9.201909}), the support function $p=\frac{r^2}{g}$ is positive
on the time interval $[0, t_0]$. The evolving curve $X(\cdot, t)$ is star-shaped with respect to $O$ for
$t\in [0, t_0]$. Since $r$ is the distance from $O$ to $X(\theta, t)$, one obtains
\begin{eqnarray*}
r(\theta, t) \leq \frac{L(t)}{2} \leq \frac{L_0}{2}
\end{eqnarray*}
which gives us (\ref{eq:2.6.201909}) and enables us to %. Using (\ref{eq:2.6.201909}) one can
revise the estimate (\ref{eq:2.9.201909}) as
\begin{eqnarray}\label{eq:2.10.201909}
\left|\frac{\partial r}{\partial \theta}\right|
\leq \max\left\{\max_{\theta}\left|\frac{\partial r}{\partial \theta}(\theta, 0)\right|,
~\frac{L_0}{2},
~\frac{3L_0}{\pi}\right\}
=\max\left\{\max_{\theta}\left|\frac{\partial r}{\partial \theta}(\theta, 0)\right|,
~\frac{3L_0}{\pi}\right\}.
\end{eqnarray}
So we have done.
%Because the flow (\ref{eq:1.1.201909}) preserves the area of the evolving curve, the $L^2$-norm of $r$ is a constant
%$$\int_0^{2\pi}r^2 d\theta\equiv A_0.$$
%By the well known Sobolev Inequality, one gets that
%\begin{eqnarray*}
%\max_{\theta\in [0, 2\pi]}r(\theta, t)\leq \frac{1}{\sqrt{2\pi}}\left(\int_0^{2\pi} r^2 d\theta\right)^{1/2}
%+\sqrt{2\pi}\left(\int_0^{2\pi} (r_1)^2 d\theta\right)^{1/2}\triangleq C.
%\end{eqnarray*}
\end{proof}

It follows from Lemma \ref{lem:2.4.201909} that the support function of the evolving curve is
positive everywhere once a positive lower bound of $r$ is given. So the flow (\ref{eq:1.1.201909}) preserves
star-shapedness of the evolving curve under the condition of Lemma \ref{lem:2.4.201909}. Now one can
conclude that:
\begin{lemma}\label{lem:2.5.201909}
If the flow (\ref{eq:1.1.201909}) does not blow up in the time interval $[0, \omega)$
and $r_{\min}(t)>0$ for $t\in [0, \omega)$ then the
point $O$ is the star center of every evolving curve $X(\cdot, t)$.
\end{lemma}

\section{Global existence}\label{sec:3.201909}

In this section, it is shown that Gage's area-preserving flow (\ref{eq:1.1.201909}) exists in the time interval
$[0, +\infty)$ if the initial smooth curve $X_0$ is centrosymmetric and star-shaped with respect to the
origin $O$.

\subsection{Star-shaped curves under the CSF}\label{subsec:3.1.201909}

The popular curve shortening flow with a smooth, closed and embedded initial curve $X_0$ is defined by
\begin{equation}\label{eq:3.1.201909}
\left\{\begin{array}{l}
\frac{\partial Y}{\partial t}(\varphi, t)=\widetilde{\kappa}(\varphi, t)\widetilde{N}(\varphi, t)
                  \ \ \ \text{in} \ \ S^1\times (0, \omega),\\
Y(\varphi, 0)= X_0(\varphi) \ \  \ \ \ \  \text{on} \ \ S^1,
\end{array} \right.
\end{equation}
where $\widetilde{\kappa}(\varphi, t)$ is the relative curvature with respect to the Frenet frame
$\{\widetilde{T}, \widetilde{N}\}$. Grayson theorem \cite{Grayson-1987} asserts that
the evolving curve $Y(\cdot, t)$ is smooth, preserves its embeddedness and becomes convex on the time interval
$\left[0, \frac{A_0}{2\pi}\right)$, where $A_0$ is the area bounded by $X_0$. Gage-Hamilton theorem
\cite{Gage-1983, Gage-1984, Gage-Hamilton-1986} says that a closed convex curve $X_0$ evolving under
the CSF (\ref{eq:3.1.201909}) becomes asymptotically circular as $t\rightarrow \frac{A_0}{2\pi}$.

If $X_0$ is star-shaped with respect to the origin $O$, it follows from the
continuity of the evolving curve that there exists $t_0>0$ such that $Y(\cdot, t)$ is
star-shaped with respect to $O$ for $t\in [0, t_0)$. So one can add a proper tangent component
to the flow (\ref{eq:3.1.201909})
\begin{equation}\label{eq:3.2.201909}
\left\{\begin{array}{l}
\frac{\partial Y}{\partial t}(\varphi, t)=\widetilde{\alpha} \widetilde{T}+\widetilde{\kappa}\widetilde{N}
                  \ \ \ \text{in} \ \ S^1\times (0, \omega),\\
Y(\varphi, 0)= X_0(\varphi) \ \  \ \ \ \  \text{on} \ \ S^1
\end{array} \right.
\end{equation}
to make the polar angle $\theta$ of $Y(\cdot, t)$ independent of time. Now let us parameterize $Y(\cdot, t)$ by $\theta$
and set $Y(\theta, t)=\rho(\theta, t)P(\theta)$ where $P(\theta)=(\cos\theta, \sin\theta)$.
As is well known that the solution to the flow (\ref{eq:3.1.201909}) differs from that
of the flow (\ref{eq:3.2.201909}) by a reparametrization and a Euclidean translation.
The evolution equation of the radial function $\rho(\theta, t)$ is
\begin{eqnarray}\label{eq:3.3.201909}
\frac{\partial \rho}{\partial t}=\frac{1}{(g_\rho)^2}\frac{\partial^2 \rho}{\partial \theta^2}
-\frac{2}{\rho(g_\rho)^2}\left(\frac{\partial \rho}{\partial \theta}\right)^2
-\frac{\rho}{(g_\rho)^2},
\end{eqnarray}
where $g_\rho = \sqrt{\rho^2 +\left(\frac{\partial \rho}{\partial \theta}\right)^2}$ is the metric
of $Y(\theta, t)$.
\begin{lemma}\label{lem:3.1.201909}
Suppose the initial smooth curve $X_0$ is star-shaped and centrosymmetric with respect to $O$.
Under the CSF (\ref{eq:3.2.201909}), if the evolving curve $Y(\cdot, t)$ is star-shaped with respect to $O$
 then it is centrosymmetric with respect to $O$.
\end{lemma}
\begin{proof}
For $t_*\in \left[0, \frac{A_0}{2\pi}\right)$, the Gage-Hamilton-Grayson theorem tells us that
there exist constants $M_i (t_*)>0$ such that
\begin{eqnarray}\label{eq:3.4.201909}
\left|\frac{\partial^i \rho}{\partial \theta^i}(\theta, t)\right| \leq M_i(t_*), \ \ \ \
(\theta, t)\in [0, 2\pi]\times [0, t_*],\ \  i=1, 2, \cdots.
\end{eqnarray}

Define $\widetilde{\rho}(\theta, t) = \rho (\theta+\pi, t)$ and
$\varphi_1(\theta, t)= \widetilde{\rho}(\theta, t) - \rho (\theta, t)$. Since $X_0$ is centrosymmetric,
$\varphi_1(\theta, 0)\equiv 0$. By the equation (\ref{eq:3.3.201909}),
$\varphi_1(\theta, t)$ evolves according to
\begin{eqnarray}
\frac{\partial \varphi_1}{\partial t} &=& \frac{1}{(g_{\widetilde{\rho}})^2}\frac{\partial^2 \varphi_1}{\partial \theta^2}
+\frac{1}{(g_{\widetilde{\rho}})^2}\frac{\partial^2 \rho}{\partial \theta^2}
-\frac{1}{(g_{\rho})^2}\frac{\partial^2 \rho}{\partial \theta^2}
+\frac{2}{\rho(g_\rho)^2}
\left[\left(\frac{\partial \rho}{\partial \theta}\right)^2 - \left(\frac{\partial \widetilde{\rho}}{\partial \theta}\right)^2\right]
\nonumber\\
&&+\frac{2}{\rho(g_\rho)^2} \left(\frac{\partial \widetilde{\rho}}{\partial \theta}\right)^2
-\frac{2}{\widetilde{\rho}(g_{\widetilde{\rho}})^2} \left(\frac{\partial \widetilde{\rho}}{\partial \theta}\right)^2
+\frac{\rho - \widetilde{\rho}}{(g_\rho)^2} + \frac{\widetilde{\rho}}{(g_\rho)^2} - \frac{\widetilde{\rho}}{(g_{\widetilde{\rho}})^2}
\nonumber\\
&=& \frac{1}{(g_{\widetilde{\rho}})^2}\frac{\partial^2 \varphi_1}{\partial \theta^2}
-\frac{1}{(g_\rho)^2 (g_{\widetilde{\rho}})^2}\frac{\partial^2 \rho}{\partial \theta^2}
           \left(\frac{\partial \rho}{\partial \theta} + \frac{\partial \widetilde{\rho}}{\partial \theta}\right)
           \frac{\partial \varphi_1}{\partial \theta}
-\frac{2}{\rho(g_\rho)^2}\left(\frac{\partial \rho}{\partial \theta}
+\frac{\partial \widetilde{\rho}}{\partial \theta}\right)\frac{\partial \varphi_1}{\partial \theta}
\nonumber\\
&& +\frac{2}{\rho(g_\rho)^2(g_{\widetilde{\rho}})^2}\left(\frac{\partial \widetilde{\rho}}{\partial \theta}\right)^2 \left(\frac{\partial \rho}{\partial \theta}
+\frac{\partial \widetilde{\rho}}{\partial \theta}\right)\frac{\partial \varphi_1}{\partial \theta}
+\frac{\widetilde{\rho}}{(g_\rho)^2 (g_{\widetilde{\rho}})^2}
    \left(\frac{\partial \rho}{\partial \theta} + \frac{\partial \widetilde{\rho}}{\partial \theta}\right)\frac{\partial \varphi_1}{\partial \theta}
\nonumber\\
&& -\frac{\widetilde{\rho} + \rho}{(g_\rho)^2(g_{\widetilde{\rho}})^2} \frac{\partial^2 \rho}{\partial \theta^2} \varphi_1
+\frac{2(\widetilde{\rho} + \rho)}{\rho (g_\rho)^2(g_{\widetilde{\rho}})^2} \left(\frac{\partial \widetilde{\rho}}{\partial \theta}\right)^2 \varphi_1
+\frac{2}{\rho \widetilde{\rho} (g_{\widetilde{\rho}})^2} \left(\frac{\partial \widetilde{\rho}}{\partial \theta}\right)^2 \varphi_1
\nonumber\\
&&
-\frac{\varphi_1}{(g_{\rho})^2} +\frac{\widetilde{\rho}(\widetilde{\rho} + \rho)}{(g_\rho)^2(g_{\widetilde{\rho}})^2}\varphi_1.
\label{eq:3.5.201909}
\end{eqnarray}
The equation (\ref{eq:3.5.201909}) is linear with smooth coefficients (see (\ref{eq:3.4.201909})) and zero initial value.
By the uniqueness of the solution to linear parabolic equations, one has
\begin{eqnarray} \label{eq:3.6.201909}
\varphi_1(\theta, t) \equiv 0.
\end{eqnarray}
So the evolving curve $Y(\cdot, t)$ is centrosymmetric with respect to $O$ for $t\in [0, t_*]$.
The proof is completed by the arbitrary choice of $t_*$.
\end{proof}

\begin{lemma}\label{lem:3.2.201909}
Let the initial smooth curve $X_0$ be star-shaped and centrosymmetric with respect to $O$, then
under the CSF (\ref{eq:3.2.201909}), the evolving curve $Y(\cdot, t)$ is star-shaped.
\end{lemma}
\begin{proof}
Suppose that there is a time $t_*\in [0, \frac{A_0}{2\pi})$ such that $Y(\cdot, t)$ is star-shaped
with respect to $O$ for $t\in [0, t_*)$ but $Y(\cdot, t_*)$ is not so.% $O$.
One can claim that
\begin{eqnarray}\label{eq:3.7.201909}
\rho(\theta, t)>0, \ \ \ \ (\theta, t) \in [0, 2\pi] \times [0, t_*),
\end{eqnarray}    %holds for $(\theta, t) \in [0, 2\pi] \times [0, t_*)$ and ???????
and
\begin{eqnarray}\label{eq:3.8.201909}
\rho_{\min}(t_*) :=\min\{\rho(\theta, t_*)| \theta \in [0, 2\pi]\}=0.
\end{eqnarray}
In fact, $Y(\cdot, t)$ is star-shaped with respect to $O$ for $t\in [0, t_*)$ so (\ref{eq:3.7.201909}) holds.
If (\ref{eq:3.8.201909}) does not hold then, for some $\delta >0$, one obtains
\begin{eqnarray}\label{eq:3.9.201909}
\rho_{\min}(t_*) \geq \delta.
\end{eqnarray}
By the evolution equation (\ref{eq:3.3.201909}), %one can calculate the evolution equation
the quantity $v=\frac{1}{2}\left(\frac{\partial \rho}{\partial \theta}\right)^2$ evolves according to
\begin{eqnarray*}
\frac{\partial v}{\partial t} &=& \frac{1}{(g_\rho)^2}\frac{\partial^2 v}{\partial \theta^2}
-\frac{1}{(g_\rho)^2}\left(\frac{\partial^2 \rho}{\partial \theta^2}\right)^2
-\frac{4v}{(g_\rho)^4}\left(\frac{\partial^2 \rho}{\partial \theta^2}\right)^2
-\frac{4}{\rho(g_\rho)^2}\frac{\partial \rho}{\partial \theta}\frac{\partial v}{\partial \theta}
\\
&& +\frac{8 v^2}{\rho^2(g_\rho)^2} + \frac{16 v^2}{(g_\rho)^4}
+\frac{8v}{\rho(g_\rho)^4} \frac{\partial \rho}{\partial \theta} \frac{\partial v}{\partial \theta}
 -\frac{2 v}{(g_\rho)^2} + \frac{4 \rho^2 v}{(g_\rho)^4}.
\end{eqnarray*}
At the point $(\theta_*, t)$ where $v(\theta, t)$ attains $v_{\max}(t)$, one gets
\begin{eqnarray*}
\frac{\partial^2 v}{\partial \theta^2} (\theta_*, t) \leq 0, ~~  ~~
~~\frac{\partial v}{\partial \theta} (\theta_*, t) = 0,
\end{eqnarray*}
and thus %Substituting the point $(\theta_*, t)$ into the equation of $v$, one obtains
\begin{eqnarray*}
\frac{\partial v}{\partial t} (\theta_*, t)
&\leq& \frac{8 v^2}{\rho^2(g_\rho)^2}(\theta_*, t) + \frac{16 v^2}{(g_\rho)^4}(\theta_*, t)
-\frac{2 v}{(g_\rho)^2}(\theta_*, t) + \frac{4 \rho^2 v}{(g_\rho)^4}(\theta_*, t)
\\
&\leq& \frac{4v}{\rho^2}(\theta_*, t)+4+0+\frac{1}{2}
\leq \frac{4v}{\delta^2}+\frac{9}{2},
\end{eqnarray*}
where $t\in [0, t_*)$. It follows from the maximum principle that
\begin{eqnarray*}
v_{\max}(t) \leq \left(v_{\max}(0) +\frac{9}{8}\delta^2\right)e^{\frac{4}{\delta^2}t_*},
\end{eqnarray*}
i.e.,
\begin{eqnarray}\label{eq:3.10.201909}
\left|\frac{\partial \rho}{\partial \theta}\right|_{\max}(t)
\leq \sqrt{\left(2v_{\max}(0) +\frac{9}{4}\delta^2\right)e^{\frac{4}{\delta^2}t_*}}.
\end{eqnarray}
By (\ref{eq:3.9.201909}) and (\ref{eq:3.10.201909}), the support function of $Y(\cdot, t) \left(\mbox{i.e.},
p_Y (\theta, t_*) = \frac{\rho^2(\theta, t_*)}{g_\rho(\theta, t_*)}\right)$ has a positive lower bound.
So $Y(\cdot, t)$ is star-shaped with respect to $O$ for $t\in [0, t_*+\varepsilon)$. This is a contradiction to the assumption
that $Y(\cdot, t_*)$ is not star-shaped with respect to $O$. Therefore, the claim (\ref{eq:3.8.201909})
holds and there exists a $\theta_* \in [0, 2\pi]$ such that
\begin{eqnarray}\label{eq:3.11.201909}
\rho(\theta_*, t_*) = 0.
\end{eqnarray}

By the choice of $t_* \in (0, \frac{A_0}{2\pi})$, the evolving curve $Y(\cdot, t_*)$ does not blow up.
By Lemma \ref{lem:3.1.201909}, $Y(\cdot, t)$ is centrosymmetric with respect to $O$. So
\begin{eqnarray}\label{eq:3.12.201909}
\rho(\theta_*+\pi, t_*) = 0.
\end{eqnarray}
The equations (\ref{eq:3.11.201909}) and (\ref{eq:3.12.201909}) contradict to Huisken's monotonic
formula \cite{Huisken-1998} or the fact that the evolving curve keeps embedded before it shrinks to a point
(see Corollary 3.2.4 of \cite{Gage-Hamilton-1986}).
Therefore, the equation (\ref{eq:3.8.201909}) can not hold. For every $t\in \left[0, \frac{A_0}{2\pi}\right)$,
the evolving curve $Y(\cdot, t)$ is star-shaped with respect to $O$.
\end{proof}
As a direct corollary of Lemma \ref{lem:3.1.201909} and Lemma \ref{lem:3.2.201909}, we have
\begin{corollary}\label{cor:3.3.201909}
Suppose the initial smooth curve $X_0$ is star-shaped and centrosymmetric with respect to $O$.
Under the CSF (\ref{eq:3.2.201909}), the evolving curve $Y(\cdot, t)$ shrinks to the point $O$
as $t \rightarrow \frac{A_0}{2\pi}$.
\end{corollary}

In 2010, Mantegazza in his note \cite{Mantegazza-2010} showed that
a star-shaped (with respect to $O$) initial curve remains so under
the CSF till the point $O$ is contained in the open region bounded by the evolving curve.
%In the year 2015, the first author conjectured that there exists a smooth closed embedded and
%star-shaped inial curve under the CSF will lost the star-shapedness in the evolution process.

\begin{remark}\label{rem:3.4.201909}
There exist some smooth, closed, star-shaped but not embedded curves which
%such that the CSFs with these initial curves do not always preserve the star shape of evolving curves.
do not always preserve the star-shapedness of evolving curves under the CSF.
Figure \ref{fig:1} presents such a curve $X_0$ which is of positive curvature everywhere and
is star-shaped with respect to $O$, but the two small loops shrink so that
they do not intersect each other after some seconds under the CSF. This makes the evolving curve
no longer star shaped. This example also indicates that the embeddedness %condition
of the initial curve in Theorem \ref{thm:1.1.201909} is so important that it can not be omitted.
\end{remark}

\begin{figure}[tbh]
\centering
\includegraphics[scale=0.8]{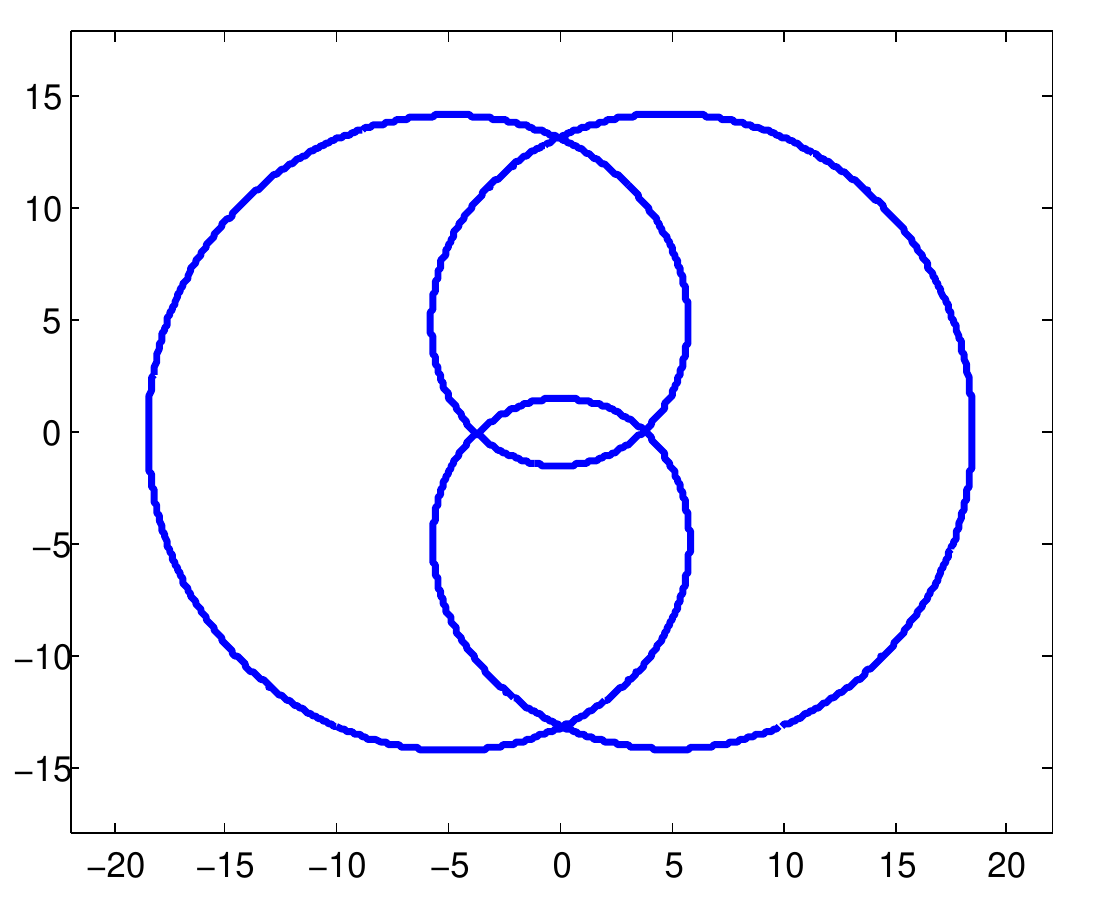}
\caption{An Immersed Star-shaped Curve $X_0$.}\label{fig:1}
\end{figure}

\subsection{Star-shapedness of the evolving curve under the flow (\ref{eq:1.1.201909})}\label{subsec:3.2.201909}

It is proved in this subsection that Gage's area-preserving flow (\ref{eq:1.1.201909}) preserves
the star-shapedness of the evolving curve if the initial smooth curve $X_0$ is both centrosymmetric
and star-shaped with respect to the origin $O$.

\begin{lemma}\label{lem:3.4.201909}
Suppose the initial smooth curve $X_0$ is star-shaped and centrosymmetric with respect to $O$ and
Gage's area-preserving flow (\ref{eq:1.1.201909}) exists on the time interval $[0, \omega)$. If the
evolving curve $X(\cdot, t)$ under (\ref{eq:1.1.201909}) is star-shaped with respect to $O$ for all
$t\in [0, \omega)$ then it is centrosymmetric with respect to $O$.
\end{lemma}
\begin{proof}
For every $t_*\in [0, \omega)$, the radial function $r(\theta, t)$ evolves according to the equation
(\ref{eq:2.5.201909}) under Gage's area-preserving flow (\ref{eq:1.1.201909}). By assumption,
$X(\cdot, 0)$ is star-shaped with respect to $O$, there exists a constant $c=c(t_*)>0$ such that
for all $(\theta, t)\in [0, 2\pi] \times [0, t_*)$,
\begin{eqnarray}\label{eq:3.13.201909}
r(\theta, t) \geq c(t_*).
\end{eqnarray}
Lemma \ref{lem:2.4.201909} tells us that (\ref{eq:2.6.201909}) and (\ref{eq:2.7.201909}) hold for $t\in [0, t_*]$.
By the classical theory of parabolic equations \cite{Lieberman-1996}, there exist constants $C_i(t_*)>0$
such that
\begin{eqnarray}\label{eq:3.14.201909}
\left|\frac{\partial^i r}{\partial \theta^i}(\theta, t)\right| \leq C_i(t_*), \ \ \ \
(\theta, t) \in [0, 2\pi] \times [0, t_*],\ \ \ i=1, 2, \cdots.
\end{eqnarray}

Set $\widetilde{r}(\theta, t):=r(\theta+\pi, t)$. By (\ref{eq:2.5.201909}), the function
$\varphi_2(\theta, t):= \widetilde{r}(\theta, t)- r(\theta, t)$ satisfies a
linear, uniformly parabolic equation which has smooth coefficients (similar to
the equation (\ref{eq:3.5.201909})) and involves a non-local quantity $L=L(t)$ of the curve,
\begin{eqnarray*}
\frac{\partial \varphi_2}{\partial t}
&=& \frac{1}{(g_{\widetilde{r}})^2}\frac{\partial^2 \varphi_2}{\partial \theta^2}
-\frac{1}{g^2 (g_{\widetilde{r}})^2}\frac{\partial^2 r}{\partial \theta^2}
           \left(\frac{\partial r}{\partial \theta} + \frac{\partial \widetilde{r}}{\partial \theta}\right)
           \frac{\partial \varphi_2}{\partial \theta}
-\frac{2}{r g^2}\left(\frac{\partial r}{\partial \theta}
+\frac{\partial \widetilde{r}}{\partial \theta}\right)\frac{\partial \varphi_2}{\partial \theta}
\nonumber\\
&& +\frac{2}{rg^2(g_{\widetilde{r}})^2}\left(\frac{\partial \widetilde{r}}{\partial \theta}\right)^2 \left(\frac{\partial r}{\partial \theta}
+\frac{\partial \widetilde{r}}{\partial \theta}\right)\frac{\partial \varphi_2}{\partial \theta}
+\frac{\widetilde{r}}{g^2 (g_{\widetilde{r}})^2}
    \left(\frac{\partial r}{\partial \theta} + \frac{\partial \widetilde{r}}{\partial \theta}\right)\frac{\partial \varphi_2}{\partial \theta}
\nonumber\\
&& -\frac{\widetilde{r} + r}{g^2(g_{\widetilde{r}})^2} \frac{\partial^2 r}{\partial \theta^2} \varphi_2
+\frac{2(\widetilde{r} + r)}{r g^2(g_{\widetilde{r}})^2} \left(\frac{\partial \widetilde{r}}{\partial \theta}\right)^2 \varphi_2
+\frac{2}{r \widetilde{r} (g_{\widetilde{r}})^2} \left(\frac{\partial \widetilde{r}}{\partial \theta}\right)^2 \varphi_2
\nonumber\\
&&
-\frac{\varphi_2}{(g_{r})^2}
+\frac{\widetilde{r}(\widetilde{r} + r)}{g^2(g_{\widetilde{r}})^2}\varphi_2
+\frac{2\pi(\widetilde{r}+r)}{\widetilde{r}L(g_{\widetilde{r}}+g)}\varphi_2
\nonumber\\
&&
+\frac{2\pi}{\widetilde{r} L(g+g_{\widetilde{r}})} \left(\frac{\partial \widetilde{r}}{\partial \theta}
                                + \frac{\partial r}{\partial \theta}\right) \frac{\partial \varphi_2}{\partial \theta}
-\frac{2\pi g}{\widetilde{r} rL}\varphi_2,
\end{eqnarray*}
where $g_{\widetilde{r}}:=\sqrt{\widetilde{r}^2+ \left(\frac{\partial \widetilde{r}}{\partial \theta}\right)^2}$.
Since $\varphi_2(\theta, 0)\equiv 0$, one can obtain %it follows from the uniqueness of the solution to linear parabolic equations that
%\begin{eqnarray*}
$\varphi_2(\theta, t) \equiv 0,$%\end{eqnarray*}
that is to say, the evolving curve $X(\cdot, t)$ is symmetric with respect to $O$.
\end{proof}

\begin{corollary}\label{cor:3.5.201909}
Suppose the initial smooth curve $X_0$ is star-shaped and centrosymmetric with respect to $O$.
If the evolving curve $X(\cdot, t)$ under Gage's area-preserving flow (\ref{eq:1.1.201909}) is star-shaped
then $O$ is one of its star centers.
\end{corollary}
\begin{proof}
Suppose Gage's area-preserving flow (\ref{eq:1.1.201909}) with initial $X_0$ exists on a time
interval $[0, \omega)$. By continuity, there exists a small $t_0\in [0, \omega)$ such that $X(\cdot, t)$ is
star-shaped with respect to $O$ if $t<t_0$.  Lemma \ref{lem:3.4.201909} implies that $X(\cdot, t)$ is
symmetric with respect to $O$ if $t<t_0$.

Suppose there exists a $t_*\in (0, \omega)$ such that (i) $X(\cdot, t_*)$ is star-shaped with respect to
some point but not %star-shaped with respect to
the origin $O$ and (ii) $X(\cdot, t)$ is star-shaped with respect to $O$ for all $t\in [0, t_*)$.
By Lemma \ref{lem:3.4.201909} and the continuity of $X(\cdot, t)$, $X(\cdot, t_*)$ is centrosymmetric with
respect to $O$. So one can conclude that
\begin{eqnarray*}
r_{\min}(t_*) > 0.
\end{eqnarray*}
Otherwise, $X(\cdot, t_*)$ is not star-shaped with respect to any point because it is centrosymmetric with
respect to $O$.

Let $P$ be a star center of $X(\cdot, t_*)$, then the symmetry of the curve implies that $-P$ is also
a star center (see \cite{Smith-1968}). Since the set of all star centers of $X(\cdot, t_*)$ is
convex, $O$ is one of its star centers, which leads to a contradiction.
\end{proof}

Next we shall show that Gage's area-preserving flow (\ref{eq:1.1.201909}) preserves star-shapedness of the
evolving curve.

\begin{lemma}\label{lem:3.6.201909}
Suppose the initial smooth curve $X_0$ is star-shaped and centrosymmetric with respect to $O$ and
Gage's area-preserving flow (\ref{eq:1.1.201909}) exists on a time interval $[0, \omega)$.  Then the
evolving curve $X(\cdot, t)$ is star-shaped with respect to $O$ for all $t\in [0, \omega)$.
\end{lemma}
\begin{proof}
Suppose there is a $t_*\in (0, \omega)$ such that $X(\cdot, t)$ is star-shaped for $t\in [0, t_*)$
but $X(\cdot, t_*)$ is not so. It follows from Corollary \ref{cor:3.5.201909}
that $X(\cdot, t)$ is star-shaped with respect to $O$
for all $t\in [0, t_*)$. By Lemma \ref{lem:3.4.201909}, $X(\cdot, t)$ is centrosymmetric with
respect to $O$ on the same time interval $[0, t_*)$. Let
$\varepsilon_0 = \min\left\{\frac{t_*}{2}, \frac{A_0}{4\pi}\right\}$, then $X(\cdot, t_*-\varepsilon_0)$
is a centrosymmetric and star-shaped curve with respect to $O$.

Let $X(\cdot, t_*-\varepsilon_0)$ evolve according to the CSF, then we obtain a family of smooth
curves $Y(\cdot, t)$ for $t\in \left[t_*-\varepsilon_0, t_*-\varepsilon_0+\frac{A_0}{2\pi}\right).$
Write $Y(\theta, t)= \rho(\theta, t) P(\theta)$,
where $\theta$ is the polar angle independent of $t$.
By Lemma \ref{lem:3.1.201909}, Lemma \ref{lem:3.2.201909} and Corollary \ref{cor:3.3.201909},
$Y(\cdot, t)$ is star-shaped and centrosymmetric with respect to $O$ and it shrinks to $O$
as $t$ tends to the time $\left(t_*-\varepsilon_0+\frac{A_0}{2\pi}\right)$. Therefore, there exists
a $\delta = \delta (t_*) >0$ such that
\begin{eqnarray}\label{eq:3.15.201909}
\rho(\theta, t) \geq \delta(t_*)
\end{eqnarray}
for all $(\theta, t) \in [0, 2\pi] \times \left[t_*-\varepsilon_0, t_*+\frac{3A_0}{8\pi}\right]$.

Under Gage's area-preserving flow, the radial function $r(\theta, t)$ of $X(\theta, t)$
satisfies the equation (\ref{eq:2.5.201909}) for $t\in [t_*-\varepsilon_0, t_*)$.
Set $\varphi_3 (\theta, t)= r(\theta, t) - \rho(\theta, t)$, where $(\theta, t) \in [0, 2\pi] \times [t_*-\varepsilon_0, t_*)$.
By (\ref{eq:2.5.201909}) and (\ref{eq:3.3.201909}), $\varphi_3$ evolves according to
\begin{eqnarray}
\frac{\partial \varphi_3}{\partial t}
&=& \frac{1}{g^2}\frac{\partial^2 \varphi_3}{\partial \theta^2}
-\frac{1}{(g_\rho)^2 g^2}\frac{\partial^2 \rho}{\partial \theta^2}
           \left(\frac{\partial \rho}{\partial \theta} + \frac{\partial r}{\partial \theta}\right)
           \frac{\partial \varphi_3}{\partial \theta}
-\frac{2}{\rho (g_\rho)^2}\left(\frac{\partial \rho}{\partial \theta}
+\frac{\partial r}{\partial \theta}\right)\frac{\partial \varphi_3}{\partial \theta}
\nonumber\\
&& +\frac{2}{\rho (g_\rho)^2 g^2}\left(\frac{\partial \rho}{\partial \theta}\right)^2 \left(\frac{\partial \rho}{\partial \theta}
+\frac{\partial r}{\partial \theta}\right)\frac{\partial \varphi_3}{\partial \theta}
+\frac{r}{(g_\rho)^2 g^2}
    \left(\frac{\partial \rho}{\partial \theta} + \frac{\partial r}{\partial \theta}\right)\frac{\partial \varphi_3}{\partial \theta}
\nonumber\\
&& -\frac{r + \rho}{(g_\rho)^2 g^2} \frac{\partial^2 \rho}{\partial \theta^2} \varphi_3
+\frac{2(r + \rho)}{\rho (g_\rho)^2 g^2} \left(\frac{\partial \rho}{\partial \theta}\right)^2 \varphi_3
+\frac{2}{\rho r g^2} \left(\frac{\partial \rho}{\partial \theta}\right)^2 \varphi_3
\nonumber\\
&&
-\frac{\varphi_3}{g^2}
+\frac{r(r + \rho)}{(g_\rho)^2 g^2}\varphi_3
+\frac{2\pi g}{rL}.
\label{eq:3.16.201909}
\end{eqnarray}

Noticing that $\frac{2\pi g}{rL}>0$ and (\ref{eq:3.16.201909}) is a linear parabolic equation with smooth and bounded
coefficients, the maximum principle implies that there exists a constant $\widetilde{c} >0$ such that for
$t\in [t_*-\varepsilon_0, t_*)$,
\begin{eqnarray*}
\varphi_3 (\theta, t) \geq \min\{\varphi_3 (\theta, t_*-\varepsilon_0)| \theta\in [0, 2\pi]\} e^{-\widetilde{c} t}.
\end{eqnarray*}
Since $\varphi_3 (\theta, t_*-\varepsilon_0) \equiv 0$, we have
$\varphi_3 (\theta, t) \geq 0$. Thus
\begin{eqnarray} \label{eq:3.17.201909}
r (\theta, t_*) \geq \rho(\theta, t_*) \geq \delta(t_*) >0.
\end{eqnarray}
By Lemma \ref{lem:3.4.201909} and Corollary \ref{cor:3.5.201909}, $X(\cdot, t_*)$ is star-shaped, which
contradicts with the assumption.
\end{proof}

\subsection{Extending Gage's area-preserving flow}\label{subsec:3.3.201909}
Now let us extend Gage's area-preserving flow (\ref{eq:1.1.201909}). Once the curvature of the evolving curve is
bounded on the time interval $[0, T_*)$ for any $T_*>0$, one can prove the smoothness of $X(\cdot, t)$ and
the flow (\ref{eq:1.1.201909}) can be extended globally.

Suppose the initial smooth curve $X_0$ is star-shaped and centrosymmetric with respect to $O$ and
Gage's area-preserving flow (\ref{eq:1.1.201909}) with initial $X_0$ exists on the time interval $[0, T_*)$, where
$T_*>0$ is a finite number. By Lemma \ref{lem:3.4.201909}, Corollary \ref{cor:3.5.201909} and
Lemma \ref{lem:3.6.201909}, the evolving curve $X(\cdot, t)$ is star-shaped and centrosymmetric with respect to $O$.
Let $\varepsilon_0 = \min\{\frac{T_*}{2}, \frac{A_0}{4\pi}\}$. $X(\cdot, t)$ is star-shaped
 with respect to $O$ for $t\in [0, T_*)$, so there exists a $\delta_1=\delta_1(T_*)>0$ such that $
%\begin{eqnarray*}
r (\theta, t) \geq \delta_1
%\end{eqnarray*}
$ holds for $t\in [0, T_*-\varepsilon_0]$. By the proof of Lemma \ref{lem:3.6.201909}, there exists
a $\delta_2=\delta_2 (T_*)>0$ such that $
%\begin{eqnarray*}
r (\theta, t) \geq \delta_2
$ %\end{eqnarray*}
for $t\in [T_*-\varepsilon_0, T_*+\frac{A_0}{8\pi}]$. Let $\delta = \delta (T_*)=\min\{\delta_1, \delta_2\}>0$, 
then for all $(\theta, t) \in [0, 2\pi] \times [0, T_*+\frac{A_0}{8\pi})$, one can obtain
\begin{eqnarray} \label{eq:3.18.201909}
r (\theta, t) \geq \delta >0,
\end{eqnarray}
which together with Lemma \ref{lem:2.4.201909} and Lemma \ref{lem:2.5.201909} tells us that 
if $t\in\left[0, T_*+\frac{A_0}{8\pi}\right)$, then the evolving curve $X(\cdot, t)$ is star-shaped 
with respect to $O$. So there exists an $h=h(T_*)>0$ such that the ``support function" with respect 
to $O$ satisfies 
\begin{eqnarray} \label{eq:3.19.201909}
p (\theta, t) \geq 2h(T_*) >0, \ \ \ \ 
(\theta, t) \in [0, 2\pi] \times \left[0, T_*+\frac{A_0}{8\pi}\right).
\end{eqnarray}

Under the flow (\ref{eq:1.1.201909}), the curvature evolves according to
\begin{eqnarray}\label{eq:3.20.201909}
\frac{\partial \kappa}{\partial t}=\frac{\partial^2 \kappa}{\partial s^2}+\kappa^3-\frac{2\pi}{L}\kappa^2,
\end{eqnarray}
where $s$ stands for the arc length parameter. Define
\begin{eqnarray*}
\varphi_4(s, t)= \frac{\kappa (s, t)}{p(s, t)- h(T_*)}, \ \ \ \ \ \ t\in [0, T_*),
\end{eqnarray*}
%where $t\in [0, T_*)$.
we have
\begin{eqnarray*}
&&\frac{\partial \varphi_4}{\partial s} = \frac{1}{p-h}\frac{\partial \kappa}{\partial s}
   -\frac{\kappa}{(p-h)^2}\frac{\partial p}{\partial s},\\
&& \frac{\partial^2 \varphi_4}{\partial s^2} = \frac{1}{p-h}\frac{\partial^2 \kappa}{\partial s^2}
   -\frac{2}{(p-h)^2}\frac{\partial \kappa}{\partial s} \frac{\partial p}{\partial s}
   - \frac{\kappa}{(p-h)^2}\frac{\partial^2 p}{\partial s^2}
   + \frac{2\kappa}{(p-h)^3} \left(\frac{\partial p}{\partial s}\right)^2.
\end{eqnarray*}
Noticing that the support function evolves according to
\begin{eqnarray}\label{eq:3.21.201909}
\frac{\partial p}{\partial t} &=& -\frac{\partial}{\partial t}\langle X, N \rangle
=-\left\langle \left(\kappa - \frac{2\pi}{L}\right)N, N \right\rangle
  +\left\langle X, \frac{\partial \kappa}{\partial s} T \right\rangle
\nonumber \\
&=& \frac{2\pi}{L} - \kappa+ \langle X, T\rangle \frac{\partial \kappa}{\partial s},
\end{eqnarray}
the evolution equation of $\varphi_4$ can be given by
\begin{eqnarray}\label{eq:3.22.201909}
\frac{\partial \varphi_4}{\partial t}
&=&\frac{\partial^2 \varphi_4}{\partial s^2}
   +\frac{2}{p-h}\frac{\partial p}{\partial s}\frac{\partial \varphi_4}{\partial s}
   -\frac{\kappa}{p-h}\langle X, T\rangle \frac{\partial \varphi_4}{\partial s}
   +\frac{\kappa}{(p-h)^2} \frac{\partial^2 p}{\partial s^2}
\nonumber\\
&& +(p-h)^2(\varphi_4)^3 - \frac{2\pi}{L}(p-h) (\varphi_4)^2+(\varphi_4)^2
   - \frac{2\pi}{L}\frac{\varphi_4}{p-h}
\nonumber\\
&& - \frac{\kappa^2}{(p-h)^3} \langle X, T\rangle \frac{\partial p}{\partial s}.
\end{eqnarray}
Using
\begin{eqnarray*}
\frac{\partial p}{\partial s} = \kappa \langle X, T\rangle, ~~~~
\frac{\partial^2 p}{\partial s^2} =\frac{\partial \kappa}{\partial s} \langle X, T\rangle
+\kappa-\kappa^2 p,
\end{eqnarray*}
we can compute that
\begin{eqnarray} \label{eq:3.23.201909}
\frac{\kappa}{(p-h)^2} \frac{\partial^2 p}{\partial s^2}
- \frac{\kappa^2}{(p-h)^3} \langle X, T\rangle \frac{\partial p}{\partial s}
= \varphi_4 \langle X, T\rangle \frac{\partial \varphi_4}{\partial s} + (\varphi_4)^2
 -p(p-h)(\varphi_4)^3.
\end{eqnarray}
Substituting (\ref{eq:3.23.201909}) into the evolution equation of $\varphi_4$ can give us
\begin{eqnarray}\label{eq:3.24.201909}
\frac{\partial \varphi_4}{\partial t}
&=&\frac{\partial^2 \varphi_4}{\partial s^2}
   +\frac{2}{p-h}\frac{\partial p}{\partial s}\frac{\partial \varphi_4}{\partial s}
   -h(p-h)(\varphi_4)^3 -\frac{2\pi}{L}(p-h) (\varphi_4)^2\nonumber
\\
&&+2(\varphi_4)^2
   - \frac{2\pi}{L}\frac{\varphi_4}{p-h}.
\end{eqnarray}

As is well known that the arc length parameter $s$ depends on both space parameter and time $t$, 
one can not apply the maximum principle directly. In Section 2, the parameter $\varphi$ (independent of
$t$) is used to parametrize $X(\cdot, t)$. Now, let $g=\frac{\partial s}{\partial \varphi}$ and
suppose that for fixed $t$, $\varphi_4(\cdot, t)$ attains its maximum $(\varphi_4)_{\max} (t)$ at the 
point $(\varphi_*, t)$. Denoted by $s_* = s(\varphi_*)$, then
$$\frac{\partial \varphi_4}{\partial s} (s_*, t)
=\left(\frac{\partial \varphi_4}{\partial \varphi} \frac{1}{g}\right) (s(\varphi_*), t) =0$$
and
$$\frac{\partial^2 \varphi_4}{\partial s^2}(s_*, t)
= \left(\frac{\partial^2 \varphi_4}{\partial \varphi^2}\frac{1}{g^2}\right) (s(\varphi_*), t)
-\left(\frac{\partial \varphi_4}{\partial \varphi} \frac{\partial g}{\partial \varphi}
\frac{1}{g^3}\right) (s(\varphi_*), t)
\leq 0.$$
By (\ref{eq:3.24.201909}),
\begin{eqnarray*}
\frac{\partial \varphi_4}{\partial t}(s_*, t)
&\leq& -h(p-h)(\varphi_4(s_*, t))^3 - (p-h)\frac{2\pi}{L} (\varphi_4(s_*, t))^2
\\
&& +2(\varphi_4(s_*, t))^2 - \frac{2\pi}{L}\frac{\varphi_4(s_*, t)}{p-h}
\\
&<& -h(p-h)(\varphi_4(s_*, t))^3 +2(\varphi_4(s_*, t))^2
\\
&<& -h^2(\varphi_4(s_*, t))^3 +2(\varphi_4(s_*, t))^2.
\end{eqnarray*}
Once $\varphi_4(s_*, t)$ is greater than $\frac{2}{h^2}$, then $\frac{\partial \varphi_4}{\partial t}(s_*, t)< 0$.
By the maximum principle, we get
\begin{eqnarray}\label{eq:3.25.201909}
\varphi_4(s, t) \leq \max\left\{(\varphi_4)_{\max}(0), ~~\frac{2}{h^2(T_*)}\right\}
\end{eqnarray}
for $(s, t)\in [0, L] \times \left[0, T_*+\frac{A_0}{8\pi}\right]$. Since the ``support function" satisfies 
$p (s, t) \leq \frac{L(t)}{2} \leq \frac{L(0)}{2}$, we have
\begin{eqnarray}\label{eq:3.26.201909}
\kappa(s, t) \leq \left(\frac{L(0)}{2}-h(T_*)\right) \cdot\max\left\{(\varphi_4)_{\max}(0), ~~\frac{2}{h^2(T_*)}\right\}.
\end{eqnarray}
Similarly, if $\varphi_4$ attains $(\varphi_4)_{\min}(t)$ at a point $(s_*, t)$ then
$\frac{\partial \varphi_4}{\partial s} (s_*, t)= 0, ~~\frac{\partial^2 \varphi_4}{\partial s^2}(s_*, t) \geq 0$.
Once $\varphi_4(s_*, t) \leq -\frac{2\pi}{\sqrt{4\pi A_0} h(T_*)}$, we have
\begin{eqnarray*}
\frac{\partial \varphi_4}{\partial t}(s_*, t)
&\geq& -h(p-h)(\varphi_4(s_*, t))^3 - (p-h)\frac{2\pi}{L(t)} (\varphi_4(s_*, t))^2
\\
&& +2(\varphi_4(s_*, t))^2
   - \frac{2\pi}{L(t)}\frac{\varphi_4(s_*, t)}{p-h}
\\
&\geq& -\left(h \varphi_4(s_*, t) +\frac{2\pi}{L(t)}\right) (p-h)(\varphi_4(s_*, t))^2 >0,
\end{eqnarray*}
and
$$
\varphi_4(s, t) \geq \min\left\{(\varphi_4)_{\min}(0), ~-\frac{2\pi}{\sqrt{4\pi A_0} h(T_*)}\right\}
$$
for $t\in [0, L] \times \left[0, T_*+\frac{A_0}{8\pi}\right]$. Therefore we can get
\begin{eqnarray}\label{eq:3.26.201909}
\kappa(s, t) \geq
\left(\frac{L(0)}{2}-h(T_*)\right) \cdot\min \left\{(\varphi_4)_{\min}(0), ~-\frac{2\pi}{\sqrt{4\pi A_0} h(T_*)}\right\}.
\end{eqnarray}
%Therefore, one can conclude that:
\begin{lemma}\label{lem:3.7.201909}
Let the initial smooth curve $X_0$ be star-shaped and centrosymmetric with respect to $O$ and
Gage's area-preserving flow (\ref{eq:1.1.201909}) exist on a time interval $[0, T_*)$ for some $T_*>0$, then
the curvature of $X(\cdot, t)$ is bounded on the time interval $[0, T_*]$ uniformly.
\end{lemma}

As a consequence of Lemma \ref{lem:3.7.201909}, one can extend Gage's area-preserving flow (\ref{eq:1.1.201909}) globally.
\begin{corollary}\label{cor:3.8.201909}
Gage's area-preserving flow (\ref{eq:1.1.201909}) %with initial $X_0$
exists on the time interval $[0, +\infty)$ if the initial smooth curve $X_0$ is star-shaped and centrosymmetric with
respect to $O$.
\end{corollary}
\begin{proof}
Let $T_*$ be a positive number which can be chosen arbitrarily. By Lemma \ref{lem:2.4.201909}, both $r$ and
$\frac{\partial r}{\partial \theta}$ are bounded uniformly on the time interval $[0, T_*]$.
By Lemmas (\ref{lem:3.1.201909})-(\ref{lem:3.7.201909}), $\left|\frac{\partial^2 r}{\partial \theta^2}\right|$
is also bounded uniformly on the same interval $[0, T_*]$.
All the higher derivatives $\frac{\partial^i r}{\partial \theta^i} (i\geq 3)$ satisfy linear parabolic equations, %so these derivatives
they also have uniform bounds. Therefore, the evolving curve is smooth on any finite time interval, and the flow
(\ref{eq:1.1.201909}) with initial $X_0$ exists globally.
\end{proof}

\section{Convergence and a convexity theorem}

In this section, it is shown that Gage's area preserving flow can deform every smooth,
closed and embedded curve into a convex one if the flow exists on the time interval $[0, +\infty)$.

Since we have proved that the flow (\ref{eq:1.1.201909}) with inial curve $X_0$ being smooth, centrosymmetric and star-shaped
does not blow up at any finite time in Section 3,
the main result in this paper can be obtained by the following theorem immediately.
\begin{theorem}\label{thm:4.1.201909}
Given a smooth, embedded and closed initial curve $X_0$, if the flow (\ref{eq:1.1.201909})
does not blow up in the time interval $[0, +\infty)$, then the evolving curve
converges to a finite circle as time goes to infinity.
\end{theorem}

In order to prove this result, we need do some preparations. %Now we use the arc length parameter for the evolving curve. Under the flow
%\begin{equation}\label{eq:4.1.201909}
%\left\{\begin{array}{l}
%\frac{\partial X}{\partial t}(s, t)=(\kappa-\frac{2\pi}{L})N \ \ \ \text{in} \ \ [0, L(t)] \times (0, \omega),\\
%X(s, 0)= X_0(s) \ \  \ \ \ \  \text{on} \ \ [0, L(t)],
%\end{array} \right.
%\end{equation}
%the metric $g$ evolves according to the following
%\begin{eqnarray*}
%\frac{\partial g}{\partial t} = -\kappa(\kappa-\frac{2\pi}{L})g.
%\end{eqnarray*}
%Using this equation, one can get the evolution of the length:
%\begin{eqnarray*}
%\frac{d L}{dt}=-\int_0^L \kappa^2 ds+\frac{4\pi^2}{L}.
%\end{eqnarray*}
%and the relation of $\frac{\partial}{\partial t}$ and $\frac{\partial}{\partial s}$:
%\begin{eqnarray*}
%\frac{\partial}{\partial t}\frac{\partial}{\partial s}=\frac{\partial}{\partial s}\frac{\partial}{\partial t}+\kappa(\kappa-\frac{2\pi}{L}) \frac{\partial}{\partial s}
%\end{eqnarray*}
%Furthermore, we can obtain the evolution equation of the curvature:
%\begin{eqnarray}\label{eq:4.2.201909}
%\frac{\partial \kappa}{\partial t}=\frac{\partial^2 \kappa}{\partial s^2}+\kappa^3-\frac{2\pi}{L}\kappa^2.
%\end{eqnarray} Before the proof of this theorem, two lemmas will be established.
From now on, we use the subindex to stand for partial derivatives,
such as $\kappa_t=\frac{\partial \kappa}{\partial t}, \kappa_s=\frac{\partial \kappa}{\partial s},
\kappa_{ss}=\frac{\partial^2 \kappa}{\partial s^2}, \cdots$.
%By the evolution equation of $\kappa$, the elastic energy of the evolving curve satisfies
%that
%\begin{eqnarray*}
%\frac{d}{dt}\int_0^L \kappa^2 ds=-2\int_0^L (\kappa_s)^2 ds +\int_0^L \kappa^4 ds-\frac{2\pi}{L}\int_0^L \kappa^3 ds.
%\end{eqnarray*}
The length of the evolving curve is decreasing during the evolution process and bounded from below by 
$\sqrt{4\pi A_0}$. Following the parlance from \cite{Grayson-1987}, the time derivative
of the length must approach zero at an $\varepsilon$-dense set of sufficiently large time.
Now one can use this fact in the following special case.
Denote by $I_n$ the time interval $[n-\frac{1}{10}, n+\frac{1}{10}]$, where $n$ is a natural number.
For a positive number $\varepsilon$, all of the intervals which satisfy that
$\inf_{I_n}\int_0^L \kappa^2 ds-\frac{4\pi^2}{L}\geq \varepsilon$
consists of a finite set. So there exists an $N>0$ such that whenever $n>N$ we have $t_n\in I_n$ satisfying
\begin{eqnarray}\label{eq:4.1.201909}
\int_0^L \kappa(\cdot, t_n)^2 ds-\frac{4\pi^2}{L(t_n)}\leq \varepsilon.
\end{eqnarray}

\begin{lemma}\label{lem:4.2.201909}
The $L^2$ norm of the difference $(\kappa-\frac{2\pi}{L})$ converges to zero as $t\rightarrow \infty$.
\end{lemma}
\begin{proof}
Integrating the evolution equation of $L$ can give us 
$$L(t)-L(0)=-\int_0^t\int_0^L \left(\kappa-\frac{2\pi}{L}\right)^2 dsdt.$$
Letting $t$ tend to infinity can yield 
\begin{eqnarray}\label{eq:4.2.201909}
\int_0^\infty\int_0^L \left(\kappa-\frac{2\pi}{L}\right)^2 dsdt < L(0).
\end{eqnarray}
We consider the time derivative of the integral
$\int_0^L \left(\kappa-\frac{2\pi}{L}\right)^2 ds$:
\begin{eqnarray}
\frac{d}{dt}\int_0^L \left(\kappa-\frac{2\pi}{L}\right)^2 ds
&=& -2\int_0^L (\kappa_s)^2 ds + \int_0^L \kappa^2\left(\kappa-\frac{2\pi}{L}\right)^2 ds
\nonumber\\
&&
+\frac{2\pi}{L}\int_0^L \kappa\left(\kappa-\frac{2\pi}{L}\right)^2 ds. \label{eq:4.3.201909}
\end{eqnarray}
One only needs to deal with the case that the evolving curve is not convex.
In this case, $\inf_{s} \kappa^2=0$.  Otherwise, the convex curve $X(\cdot, t)$ will converge to
a finite circle (see \cite{Gage-1983}) and the proof is completed. Since
%$$\sup_{s} \kappa^2 \leq \left(\inf_{s} \kappa^2+\int_0^L |\kappa_s| ds\right)^2\leq L \int_0^L (\kappa_s)^2 ds,$$
$$\sup_{s} \kappa^2 \leq L \int_0^L (\kappa_s)^2 ds,$$
one obtains
$$-\int_0^L (\kappa_s)^2 ds\leq -\frac{1}{L}\sup_{s} \kappa^2.$$
By the Cauchy-Schwartz inequality,
\begin{eqnarray*}
\left|\int_0^L \kappa\left(\kappa-\frac{2\pi}{L}\right)^2 ds\right|
&\leq& \sqrt{\int_0^L \kappa^2\left(\kappa-\frac{2\pi}{L}\right)^2 ds} \cdot \sqrt{\int_0^L \left(\kappa-\frac{2\pi}{L}\right)^2 ds}
\\
&\leq& \sqrt{\int_0^L \sup_{s} \kappa^4 ds} \cdot \sqrt{\int_0^L \left(\kappa-\frac{2\pi}{L}\right)^2 ds}.
\\
&=& \sup_{s} \kappa^2 \cdot \sqrt{L(t)} \cdot \sqrt{\int_0^L \left(\kappa-\frac{2\pi}{L}\right)^2 ds}.
\end{eqnarray*}
Taking this estimate into the equation (\ref{eq:4.3.201909}) can yield
\begin{eqnarray*}
\frac{d}{dt}\int_0^L\left(\kappa-\frac{2\pi}{L}\right)^2 ds
&\leq& -\frac{2}{L(t)}\sup_{s} \kappa^2
+ \sup_{s} \kappa^2 \int_0^L \left(\kappa-\frac{2\pi}{L}\right)^2 ds
\\
&&
+ \sup_{s} \kappa^2 \cdot \frac{2\pi}{\sqrt{L(t)}} \cdot \sqrt{\int_0^L \left(\kappa-\frac{2\pi}{L}\right)^2 ds}
\\
&\leq& \sup_{s} \kappa^2 \left(-\frac{2}{L_0}+  \int_0^L\left(\kappa-\frac{2\pi}{L}\right)^2 ds
+\frac{2\pi}{\sqrt{L_\infty}} \sqrt{\int_0^L \left(\kappa-\frac{2\pi}{L}\right)^2 ds}\right),
\end{eqnarray*}
where $L_\infty = \sqrt{4\pi A}$. Choose $\varepsilon>0$ in (\ref{eq:4.1.201909}) small enough
such that
$$-\frac{2}{L_0}+ \frac{3}{2}\varepsilon + \frac{2\pi}{\sqrt{L_\infty}}\sqrt{\frac{3\varepsilon}{2}} \leq 0.$$
 We now claim that $\int_0^L(\kappa-\frac{2\pi}{L})^2 ds < \frac{3\varepsilon}{2}$.

Firstly, assume that at time $t=t_n$, we have
$$\int_0^L(\kappa-\frac{2\pi}{L})^2 ds \leq \varepsilon.$$
Secondly, %consider the time
for $t>t_n$, suppose there exists a $t_*\in \left(t_n, t_n+\frac{11}{10}\right]$ such that
\begin{eqnarray*}
\int_0^L \left(\kappa(\cdot, t_*)-\frac{2\pi}{L(t_*)}\right)^2 ds= \frac{3\varepsilon}{2}
\end{eqnarray*}
and
\begin{eqnarray*}
\int_0^L \left(\kappa-\frac{2\pi}{L}\right)^2 ds < \frac{3\varepsilon}{2}
\end{eqnarray*}
for all $t\in [t_n, t_*)$. In the time interval $(t_n, t_*)$, we have
\begin{eqnarray*}
\frac{d}{dt}\int_0^L\left(\kappa-\frac{2\pi}{L}\right)^2 ds &\leq& \sup_{s} \kappa^2
\left(-\frac{2}{L_0}+ \frac{3}{2}\varepsilon + \frac{2\pi}{\sqrt{L_\infty}}\sqrt{\frac{3\varepsilon}{2}} \right)
\leq 0.
\end{eqnarray*}
Thus
\begin{eqnarray*}
\int_0^L \left(\kappa(\cdot, t_*)-\frac{2\pi}{L(t_*)}\right)^2 ds &\leq& \int_0^L
\left(\kappa(\cdot, t_n)-\frac{2\pi}{L(t_n)}\right)^2 ds < \frac{3\varepsilon}{2},
\end{eqnarray*}
which contradicts with the definition of $t_*$. For any $\varepsilon>0$, there exists an $N>0$
such that $t>N$ implies that $\int_0^L(\kappa-\frac{2\pi}{L})^2 ds < \frac{3}{2}\varepsilon$.
That is to say, one has the limit
$$\lim_{t\rightarrow \infty} \int_0^L \left(\kappa-\frac{2\pi}{L}\right)^2 ds =0. $$
\end{proof}

The integral of $\int_0^L (\kappa-\frac{2\pi}{L})^2 ds$ in the equation (\ref{eq:4.3.201909}) is bounded in the time interval $[0, \infty)$.
Next, the $L^2$-norm of $(\kappa-\frac{2\pi}{L})_s$ will be estimated.

\begin{lemma}\label{lem:4.3.201909}
Under the assumption of Theorem \ref{thm:4.1.201909}, we have the limit
$$\lim_{t\rightarrow \infty} \int_0^L (\kappa_s)^2 ds =0. $$
\end{lemma}
\begin{proof}
Denote $u=\kappa-\frac{2\pi}{L}$. It is easy to calculate %from the evolution of $\kappa$
that
$$u_t=u_{ss}+\left(u+\frac{2\pi}{L}\right)^2u-\frac{2\pi}{L^2}\int_0^L u^2 ds$$ and
$$u_{st}=u_{sss}+3\left(u+\frac{2\pi}{L}\right)uu_s+\left(u+\frac{2\pi}{L}\right)^2u_s.$$
By using these equations one can get
\begin{eqnarray}\label{eq:4.4.201909}
\frac{d}{dt}\int_0^L(u_s)^2 ds &=& -2\int_0^L (u_{ss})^2 ds +7\int_0^L u^2(u_s)^2 ds
\nonumber\\
&&  + \frac{18\pi}{L} \int_0^L u (u_s)^2 ds +\frac{8\pi^2}{L^2}\int_0^L(u_s)^2 ds.
\end{eqnarray}
If $\int_0^L (u_{s})^2 ds>C\int_0^L u^2 ds$ for some positive $C$ then from
$$\int_0^L (u_{s})^2 ds=-\int_0^L uu_{ss} ds\leq \sqrt{\frac{1}{C}\int_0^L (u_s)^2 ds\int_0^L (u_{ss})^2 ds}$$
it follows that
\begin{eqnarray}\label{eq:4.5.201909}
\int_0^L (u_s)^2 ds\leq \frac{1}{C} \int_0^L (u_{ss})^2 ds.
\end{eqnarray}
We now estimate the terms in the right hand side of the equation (\ref{eq:4.4.201909}).
\begin{eqnarray}
\int_0^L u^2(u_s)^2 ds &\leq& \int_0^L u^2 ds \cdot \sup_s(u_s)^2\leq \int_0^L u^2 ds \left(\int_0^L |u_{ss}| ds\right)^2
\nonumber\\
&\leq& \int_0^L u^2 ds L \int_0^L (u_{ss})^2 ds. \label{eq:4.6.201909}
\end{eqnarray}
Using (\ref{eq:4.5.201909})-(\ref{eq:4.6.201909}) can yield
\begin{eqnarray}
\int_0^L u(u_s)^2 ds &\leq& \frac{1}{2}\int_0^L u^2(u_s)^2 ds+\frac{1}{2} \int_0^L (u_s)^2 ds
\nonumber\\
&\leq& \frac{L}{2}\int_0^L u^2 ds \int_0^L (u_{ss})^2 ds
+ \frac{1}{2C} \int_0^L (u_{ss})^2 ds. \label{eq:4.7.201909}
\end{eqnarray}
Taking the estimates (\ref{eq:4.5.201909})-(\ref{eq:4.7.201909}) into the equation (\ref{eq:4.4.201909}), one obtains
\begin{eqnarray*}
\frac{d}{dt}\int_0^L(u_s)^2 ds
&=& \left(-2+(7L+9\pi)\int_0^L u^2 ds+ \frac{9\pi}{C L}+ \frac{8\pi^2}{CL^2}\right)\int_0^L (u_{ss})^2 ds
\\
&\leq& \left(-2+(7L_0+9\pi)\int_0^L u^2 ds+ \frac{9\pi}{C L_\infty}+ \frac{8\pi^2}{CL_\infty^2}\right)\int_0^L (u_{ss})^2 ds.
\end{eqnarray*}
Lemma \ref{lem:4.2.201909} tells us that there exists a positive $t_0$ such that for $t>t_0$, 
$$(7L_0+9\pi)\int_0^L u^2 ds \leq \frac{1}{2}.$$
If $C\geq \max\{\frac{36\pi}{L_\infty}, \frac{32\pi^2}{L_\infty^2}\}$
then $\frac{9\pi}{C L_\infty}+ \frac{8\pi^2}{CL_\infty^2} \leq1/2$ and furthermore
\begin{eqnarray*}
\frac{d}{dt}\int_0^L(u_s)^2 ds \leq -\int_0^L (u_{ss})^2 ds
\end{eqnarray*}
if $t>t_0$. For a $C^1$ and $2\pi$-periodic function $f$ defined on $\mathbb{R}$ such that $\int_0^{2\pi} f dx=0$,
Wirtinger's inequality tells us
$$\int_0^{2\pi} (f^\prime)^2 dx\geq \int_0^{2\pi} f^2 dx.$$
Let $s=\frac{L}{2\pi}\phi$, then $$u_s=u_\phi\frac{2\pi}{L} \ \ \ \mbox{and} 
\ \ \ \ u_{ss}=u_{\phi\phi}\left(\frac{2\pi}{L}\right)^2.$$
By Wirtinger's inequality, we get
\begin{eqnarray*}
\int_0^L (u_{ss})^2 ds = \left(\frac{2\pi}{L}\right)^3\int_0^{2\pi} (u_{\phi\phi})^2 d\phi
\geq \left(\frac{2\pi}{L}\right)^3 \int_0^{2\pi} (u_\phi)^2 d\phi
=\left(\frac{2\pi}{L}\right)^2 \int_0^L (u_s)^2 ds.
\end{eqnarray*}
If $C\geq \max\{\frac{36\pi}{L_\infty}, \frac{32\pi^2}{L_\infty^2}\}$, then
$$\frac{d}{dt}\int_0^L(u_s)^2 ds \leq -\left(\frac{2\pi}{L_0}\right)^2 \int_0^L (u_s)^2 ds, \ \ \ t>t_0.
$$
Thus, for $t>t_0$,
\begin{eqnarray}\label{eq:4.8.201909}
\int_0^L (u_{s})^2 ds \leq \int_0^L (u_{s}(\cdot, t_0))^2 ds\cdot \exp \left[-\left(\frac{2\pi}{L_0}\right)^2t\right].
\end{eqnarray}
For large time, the quantity $\int_0^L (u_{s})^2 ds$ either decays exponentially
or is bounded,
$$\int_0^L (u_{s})^2 ds
\leq \left(1+\max\left\{\frac{36\pi}{L_\infty}, \frac{32\pi^2}{L_\infty^2}\right\}\right)\int_0^L u^2 ds.$$
In either event, it decreases to zero.
\end{proof}

Now we go back to the proof of Theorem \ref{thm:4.1.201909}. It follows from Sobolev's
inequality that $|\kappa-\frac{2\pi}{L}|$ tends to 0 as
$t\rightarrow \infty$. Thus there is a time $T_0>0$ such that $\kappa>0$ for all $t>T_0$.
The evolving curve becomes a convex one. By the result in Gage's original paper \cite{Gage-1986} (see also the note %by Chao-Ling-Wang
\cite{Chao-Ling-Wang}), we know that the curvature of the evolving curve converges to
$\sqrt{\frac{\pi}{A_0}}$ in the $C^\infty$ metric.

In the section 4 of the paper \cite{Gao-Zhang-2019} by the first author and Zhang,
the evolving curve of GAPF is proved to converge to a fixed limiting circle not
escaping to infinity or oscillating indefinitely as $t \rightarrow +\infty$.
So the proof of Theorem \ref{thm:1.1.201909} can be completed by combining Corollary \ref{cor:3.8.201909}
and Theorem \ref{thm:4.1.201909}.

~\\
\textbf{Acknowledgments}
Laiyuan Gao is supported by the National Natural Science Foundation of China (No.11801230).
Shengliang Pan is supported by the National Natural Science Foundation of China (No.12071347).
The authors are very grateful to Professor Michael Gage for his inspiration
on the flow (\ref{eq:1.1.201909}). The first author thanks the Department of Mathematics of UC San Diego
for providing a fantastic research atmosphere because a part of this paper
was finished when he visited Professor Lei Ni.

{\bf Laiyuan Gao}

School of Mathematics and Statistics, Jiangsu Normal University.

No.101, Shanghai Road, Xuzhou City, Jiangsu Province, China.

Email: lygao@jsnu.edu.cn\\

{\bf Shengliang  Pan}

School of Mathematical Sciences, Tongji University.

No.1239, Siping Road, Shanghai, China.

Email: slpan@tongji.edu.cn \\


\begin{thebibliography}{99}
\bibliographystyle{alpha}
\addtolength{\itemsep}{-1.5ex}
\bibitem{Chao-Ling-Wang} X. L. Chao, X. R. Ling and X. L. Wang, On a planar area-preserving curvature flow.
Proc. Amer. Math. Soc. 141 (2013), 1783-1789.

\bibitem{Chou-Zhu} K. S. Chou \& X. P. Zhu, The Curve Shortening Problem, CRC Press, Boca Raton, FL, 2001.

\bibitem{Gage-1983} M. E. Gage, An isoperimetric inequality with applications to curve shortening.
Duke Math. J. 50:4 (1983), 1225-1229.

\bibitem{Gage-1984} M. E. Gage, Curve shortening makes convex curves circular.
Invent. Math. 76:2 (1984), 357-364.

\bibitem{Gage-1986}  M. E. Gage, On an area-preserving evolution equation for plane curves,
in ``Nonlinear Problems in Geometry'' (D. M. DeTurck edited), Contemp. Math. Vol.51 (1986), 51-62.

\bibitem{Gage-1993} M. E. Gage, Evolving plane curves by curvature in relative geometries,
Duke Math. J. 72(2) (1993) 441-466.

\bibitem{Gage-Hamilton-1986} M. E. Gage \& R. S. Hamilton, The heat equation shrinking convex plane curves,
J. Differentail. Geom., 23(1986), 69-96.

\bibitem{Gage-Li-1994} M. E. Gage \& Yi Li, Evolving plane curves by curvature in relative geometries. II,
Duke Math. J. 75(1) (1994) 79-98.

\bibitem{Gao-Zhang-2019} L. Y. Gao, Y. T. Zhang, On Yau's problem of evolving one curve to another: convex case,
J. Differential Equations, no.1, 266 (2019), 179-201.

\bibitem{Grayson-1987} M. A. Grayson, The heat equation shrinks embedded plane curve to round points,
J. Differential Geom. 26:2 (1987), 285-314.

\bibitem{Grayson-1989} M. A. Grayson, Shortening Embedded Curves,
Annals of Mathematics, Second Series, Vol. 129, No. 1 (Jan., 1989), pp. 71-111.

\bibitem{Gru-2007} P. M. Gruber, Convex and discrete geometry. Grundlehren der Mathematischen Wissenschaften
[Fundamental Principles of Mathematical Sciences], 336. Springer, Berlin, 2007.



\bibitem{Hs-1981} C. C. Hsiung, A First Course in Differential Geometry, Pure \& Applied Math.,
Wiley, New York, 1981.

\bibitem{Huisken-1987} G. Huisken, The volume preserving mean curvature flow.
J. reine angew. Math, 382:78(1987) 35-48.

\bibitem{Huisken-1998} G. Huisken, A distance comparison principle for evolving curves.
Asian J. Math., 2:1(1998), 127-133.


\bibitem{Kim-Kwon-2018} I. Kim and D. Kwon, Volume preserving mean curvature flow for
star-shaped sets. arXiv:1808.04922v1 [math.AP], 2018. http://arxiv.org/abs/1808.04922v1

\bibitem{Lieberman-1996} G. M. Lieberman, Second order parabolic differential eqautions,
World Scientific, Singapore 1996.

\bibitem{Mantegazza-2010} C. Mantegazza, Star-Shapedness under Mean Curvature Flow, 2010,
Unpublished Notes, http://cvgmt.sns.it/HomePages/cm/other/MCF-1star.pdf

\bibitem{Mayer-2001} U. F. Mayer, A singular example for the averaged mean curvature flow.
Experiment. Math.  no. 1, 10  (2001), 103-107.


\bibitem{Sch-2014} R. Schneider, Convex bodies: the Brunn-Minkowski theory. Second expanded edition.
Encyclopedia of Mathematics and its Applications, 151. Cambridge University Press, Cambridge, 2014.


\bibitem{PC-2018} Privite communication with Professor Michael Gage.

\bibitem{Smith-1968} C. R. Smith, A Characterization of Star-Shaped Sets.
The American Mathematical Monthly, Vol. 75, No. 4 (Apr., 1968), p. 386.


\bibitem{Wang-Wo-Yang-2018} X. L. Wang, W. F. Wo, M. Yang,
Evolution of non-simple closed curves in the area-preserving curvature flow.
Proc. Roy. Soc. Edinburgh Sect. A, no. 3, 148 (2018), 659-668.

\end{thebibliography}
\end{document}